\documentclass[12pt,a4paper]{article}

\usepackage{amsmath,amssymb}
\usepackage{amsthm}

\newtheorem{lemma}{Lemma}
\newtheorem{proposition}{Proposition}
\newtheorem{remark}{Remark}
\newtheorem{theorem}{Theorem}

\title{Discrete multitime multiple recurrence}

\author{Cristian Ghiu$^1$, Raluca Tulig\u{a}$^2$, Constantin Udri\c ste$^2$}
\date{}

\begin{document}

\maketitle

\begin{center}
{\footnotesize
$^1$University Politehnica of Bucharest,
Faculty of Applied Sciences,

Department of Mathematical Methods and Models,
Splaiul Independentei 313,

Bucharest 060042, Romania;
e-mail: crisghiu@yahoo.com

$^2$University Politehnica of Bucharest,
Faculty of Applied Sciences, Department of Mathematics-Informatics,
Splaiul Independentei 313,
Bucharest 060042, Romania;
e-mails: ralucacoada@yahoo.com; udriste@mathem.pub.ro
}
\end{center}

\begin{abstract}
The aim of our paper is to formulate and solve problems concerning multitime multiple recurrence equations.
We discuss in detail the generic properties and the existence and uniqueness of solutions. Among the general things,
we discuss in detail the cases of autonomous and non-autonomous recurrences,
highlighting in particular the theorems of existence and uniqueness of solutions.
Finally, are given interesting examples which are the analogue of arithmetic progression and the analogue of
geometric progression. The multitime multiple recurrences are required in analysis of algorithms,
computational biology, information theory, queueing theory, filters theory, statistical physics etc.
The theoretical part about them is little or not known, this being the first paper about the subject.
\end{abstract}



\section{General statements}

A multivariate recurrence relation is an equation that recursively defines a multivariate sequence,
once one or more initial terms are given: each further term of the sequence is defined as a function
of the preceding terms. Some simply defined recurrence relations can have very complex (chaotic)
behaviors, and they are a part of the field of mathematics known as nonlinear analysis.
We can use such recurrences including the Differential Transform Method to
solve completely integrable first order PDEs system with initial conditions via discretization.

In this paper we shall refer to
{\it discrete multitime multiple recurrence}
(autonomous and non-autonomous),
giving original results regarding generic properties and existence and uniqueness of solutions.
Also, we seek to provide a fairly thorough and unified exposition of efficient
recurrence relations in both univariate and multivariate settings.
The scientific sources used by us are:
filters theory \cite{Ah}, \cite{DM}, \cite{FoM}-\cite{FMa},
\cite{Hu}-\cite{Pr},
general recurrence theory \cite{HK}, \cite{BP},
\cite{E}, \cite{WCM}, our results regarding the diagonal multitime recurrence \cite{GTU}-\cite{GTUT},
and multitime dynamical systems \cite{U1}-\cite{U6}.

Let $m\geq 1$ be an integer number.
We denote ${\bf 1}=(1,1,\ldots, 1) \in \mathbb{Z}^m$. Also, for each
$\alpha \in \{ 1,2, \ldots, m \}$, we denote
$1_{\alpha}=(0,\ldots,0,1,0,\ldots, 0)\in \mathbb{Z}^m$, i.e.,
$1_{\alpha}$ has $1$ on the position $\alpha$ and $0$ otherwise.

On $\mathbb{Z}^m$, we define the relation $``\leq"$:
for $t=(t^1,\ldots, t^m)$,
$s=(s^1,\ldots, s^m)$,
\begin{equation*}
s\leq t\,\,\,\, \mbox{if}\,\,\,
s^{\alpha} \leq t^{\alpha},\,\,
\forall
\alpha \in \{ 1,2, \ldots, m \}.
\end{equation*}
One observes that $``\leq"$ is a partial order relation on $\mathbb{Z}^m$.

Let $M$ be an arbitrary nonvoid set and $t_1\in \mathbb{Z}^m$ be a fixed element. 
We consider the functions $F_\alpha \colon \big\{ t \in \mathbb{Z}^m \, \big|\,
t \geq t_1 \big\} \times M \to M$, $\alpha \in \{ 1,2, \ldots, m \}$.

We fix $t_0 \in \mathbb{Z}^m$, $t_0 \geq t_1$.
A first order multitime recurrence of the type
\begin{equation}
\label{rec.alfa}
x(t+1_\alpha) = F_\alpha (t,x(t)),
\quad
\forall t \in \mathbb{Z}^m,\,\,
t\geq t_0,\,\,
\forall \alpha \in \{1,2, \ldots, m\},
\end{equation}
is called a {\it discrete multitime multiple recurrence}.

This model of multiple recurrence can be justified by the fact that a completely integrable first order PDE system
$$\frac{\partial x^i}{\partial t^\alpha}(t)= X^i_\alpha(t,x(t)), \, t \in \mathbb{R}^m$$
can be discretized as
$$x^i(t+1_\alpha) = F^i_\alpha(t,x(t)),\, t\in \mathbb{Z}^m.$$
The initial (Cauchy) condition, for the PDE system, is translated
into initial condition for the multiple recurrence.

\begin{proposition}
\label{alfa.p1}
If for any $(t_0,x_0) \in
\big\{ t \in \mathbb{Z}^m \, \big|\, t \geq t_1 \big\}
\times M$,
there exists at least one solution
$x \colon
\big\{ t \in \mathbb{Z}^m \, \big|\,
t \geq t_0 \big\} \to M$
which verifies the recurrence {\rm (\ref{rec.alfa})} and the initial condition
$x(t_0)=x_0$, then
\begin{equation}
\label{alfap1.1}
F_\alpha(t+1_\beta,F_\beta (t,x))
=
F_\beta(t+1_\alpha,F_\alpha (t,x)),
\quad
\quad
\forall t \geq t_1,\,\,
\forall x \in M,
\end{equation}
$$
\forall \alpha, \beta \in \{1,2, \ldots, m\}.
$$
\end{proposition}

\begin{proof}
Let $t \geq t_0$. The equality
$x(t+1_\beta+1_\alpha) = x(t+1_\alpha+1_\beta)$
is equivalent to
$$
F_\alpha (t+1_\beta,x(t+1_\beta))
=
F_\beta (t+1_\alpha,x(t+1_\alpha))
$$
$$
\Longleftrightarrow
F_\alpha \big(t+1_\beta,F_\beta(t,x(t))\big)
=
F_\beta \big(t+1_\alpha, F_\alpha(t,x(t)) \big).
$$
For $t=t_0$, one obtains
$$
F_\alpha \big(t_0+1_\beta,F_\beta(t_0,x_0)\big)
=
F_\beta \big(t_0+1_\alpha, F_\alpha(t_0,x_0) \big).
$$
Since $t_0$ and $x_0$ are arbitrary, it follows the relations (\ref{alfap1.1}), $\forall \alpha, \beta$.
\end{proof}

\section{Autonomous discrete multitime\\ multiple recurrence}

Let $M$ be a nonvoid set. For any function $G \colon M \to M$, we denote
\begin{equation*}
G^{(n)}=
\left\{
\begin{array}{ll}
\underbrace{G \circ G \circ \ldots \circ G}_n, & \hbox{if}\,\,\,  n \geq 1;
\vspace{0.1 cm}\\
\qquad\,\,\, {\rm Id}_M, & \hbox{if}\,\,\,  n =0.
\end{array}
\right.
\end{equation*}

\subsection{Existence and uniqueness Theorem}

\begin{theorem}
\label{alfa.t1}
We consider the functions
$G_{\alpha} \colon M \to M$, $\alpha \in \{ 1,2, \ldots, m \}$.

\noindent
$a)$ Let $t_0 \in \mathbb{Z}^m $. If for any $x_0 \in  M$,
there exists at least one function
$$
x \colon \big\{ t \in \mathbb{Z}^m \, \big|\,
t \geq t_0 \big\} \to M,\,\,
$$
which verifies the recurrence equation
\begin{equation}
\label{ecalfat1.1}
x(t+1_\alpha) = G_\alpha (x(t)),
\quad
\forall t \geq t_0,\,\,
\forall \alpha \in \{1,2, \ldots, m\},
\end{equation}
and the initial condition $x(t_0)=x_0$, then
\begin{equation}
\label{ecalfat1.2}
G_{\alpha}\circ G_{\beta}
=
G_{\beta} \circ G_{\alpha},
\quad
\forall \alpha, \beta \in \{1,2, \ldots, m\}.
\end{equation}

\noindent
$b)$ If, for any
$\alpha, \beta \in \{1,2, \ldots, m\}$,
the relations {\rm (\ref{ecalfat1.2})} are satisfied, then,
for any $(t_0,x_0) \in \mathbb{Z}^m \times M$,
there exists a unique $m$-sequence
$x \colon \big\{ t \in \mathbb{Z}^m \, \big|\,
t \geq t_0 \big\} \to M$ which verifies the recurrence {\rm (\ref{ecalfat1.1})}
and the initial condition $x(t_0)=x_0$; this sequence is definite by the formula
\begin{equation}
\label{ecalfat1.3}
x(t)=
G_1^{(t^1-t_0^1)} \circ G_2^{(t^2-t_0^2)}  \circ
\ldots \circ G_m^{(t^m-t_0^m)}(x_0),
\quad
\forall t \geq t_0.
\end{equation}
\end{theorem}

\begin{proof}
$a)$ The equality
$x(t_0+1_\beta+1_\alpha) = x(t_0+1_\alpha+1_\beta)$
is equivalent to
$$
G_\alpha (x(t_0+1_\beta))
=
G_\beta (x(t_0+1_\alpha))
\Longleftrightarrow
G_\alpha (G_\beta (x(t_0)))
=
G_\beta (G_\alpha (x(t_0)))
$$
$$
\Longleftrightarrow
G_\alpha \circ G_\beta (x_0)
=
G_\beta \circ G_\alpha (x_0).
$$

Since $x_0$ is arbitrary,
it follows the relations (\ref{ecalfat1.2}),
$\forall \alpha, \beta$.

$b)$ Firstly we remark that any sequence of the form
(\ref{ecalfat1.3}) verifies the relations
(\ref{ecalfat1.1})
and the initial condition $x(t_0)=x_0$:
\begin{equation}
\label{ecalfat1.4}
x(t+1_\alpha) =
G_1^{(t^1-t_0^1)}
\circ
\ldots
\circ
G_{\alpha}^{(t^{\alpha}+1-t_0^{\alpha})}
\circ
\ldots \circ G_m^{(t^m-t_0^m)}(x_0);
\end{equation}
using (\ref{ecalfat1.2}) and the relation (\ref{ecalfat1.4}), it follows
\begin{equation*}
x(t+1_\alpha) =
G_{\alpha}
\circ
G_1^{(t^1-t_0^1)}
\circ
\ldots
\circ
G_{\alpha}^{(t^{\alpha}-t_0^{\alpha})}
\circ
\ldots \circ G_m^{(t^m-t_0^m)}(x_0)
=G_{\alpha}(x(t)).
\end{equation*}
The initial condition $x(t_0)=x_0$ is checked immediately.

The necessity is proved by induction after $m$, the components number of the point $t=(t^1,...,t^m)$.

For $m=1$, we have $t=t^1$ and $t_0=t_0^1$.
If $t>t_0$, then
\begin{equation*}
x(t)=x(t^1)=
G_1(x(t^1-1))
=G_1^{(2)}(x(t^1-2))=
\end{equation*}
\begin{equation*}
=
\ldots
=G_1^{(k)}(x(t^1-k))=
\ldots
=G_1^{(t^1-t_0^1)}(x(t_0^1))
=G_1^{(t^1-t_0^1)}(x_0).
\end{equation*}
For $t=t_0$, the relation $x(t)=G_1^{(t^1-t_0^1)}(x_0)$ is obvious.

Let $m\geq 2$. Suppose that the relation is true for
$m-1$ and we shall prove it for $m$.
We denote $\tilde{t}=(t^2, \ldots, t^m)$;
$\tilde{t}_0=(t_0^2, \ldots, t_0^m)$.

Let $\tilde{x}(\tilde{t})
=x(t_0^1,\tilde{t})
=x(t_0^1,t^2, \ldots, t^m)$.
If $t^1>t_0^1$, then
\begin{equation*}
x(t)=x(t^1,\tilde{t})=
G_1(x(t^1-1,\tilde{t}))
=G_1^{(2)}(x(t^1-2,\tilde{t}))=
\end{equation*}
\begin{equation*}
=
\ldots
=G_1^{(k)}(x(t^1-k,\tilde{t}))=
\ldots
=G_1^{(t^1-t_0^1)}(x(t_0^1,\tilde{t}))
=
G_1^{(t^1-t_0^1)}(\tilde{x}(\tilde{t})).
\end{equation*}
We have proved that if $t^1>t_0^1$, then
$x(t)=G_1^{(t^1-t_0^1)}(\tilde{x}(\tilde{t}))$;
the relation is verified automatically also for $t^1=t_0^1$.

For $\alpha \in \{ 2, \ldots, m \}$, we denote
$\tilde{1}_{\alpha}=(0,\ldots,0,1,0,\ldots, 0)\in \mathbb{Z}^{m-1}$;
hence $1_{\alpha}=(0,\tilde{1}_{\alpha})$.
For $\alpha \geq 2$ and $t^1=t_0^1$, the relations
(\ref{ecalfat1.1}) become
$x((t_0^1,\tilde{t})+(0,\tilde{1}_{\alpha}))
= G_\alpha (x(t_0^1,\tilde{t}))$,\, i.e.,
\begin{equation*}
\tilde{x}(\tilde{t}+\tilde{1}_{\alpha}) = G_\alpha (\tilde{x}(\tilde{t})),
\quad
\forall\, \tilde{t} \geq \tilde{t}_0,\,\,
\forall \alpha \in \{2, \ldots, m\}.
\end{equation*}
Obviously $\tilde{x}(\tilde{t}_0)
=x(t_0^1,\tilde{t}_0)=x(t_0)=x_0$.
Since $\tilde{t}$ has $m-1$ components,
from the induction hypothesis it follows
\begin{equation*}
\tilde{x}(\tilde{t})=
G_2^{(t^2-t_0^2)}  \circ
\ldots \circ G_m^{(t^m-t_0^m)}(x_0),
\quad
\forall \tilde{t} \geq \tilde{t}_0.
\end{equation*}

Consequently, for any $t\geq t_0$, we have
\begin{equation*}
x(t)=
G_1^{(t^1-t_0^1)}(\tilde{x}(\tilde{t}))
=
G_1^{(t^1-t_0^1)} \circ G_2^{(t^2-t_0^2)}  \circ
\ldots \circ G_m^{(t^m-t_0^m)}(x_0).
\end{equation*}
\end{proof}

\subsection{Extension theorems}

\begin{lemma}
\label{alfa.l1}
Let $G \colon M \to M$ be an arbitrary function
and $t_0 \in \mathbb{Z}^m$,
$\beta \in \{ 1,2, \ldots, m \}$, fixed.
If for any $x_0 \in  M$, there exists at least one function
$$
x \colon \big\{ t \in \mathbb{Z}^m \, \big|\,
t \geq t_0-1_{\beta} \big\} \to M,
$$
which verifies relation
\begin{equation}
\label{ecalfal1.1}
x(t+1_{\beta}) = G (x(t)),
\quad
\forall t \geq t_0-1_{\beta},
\end{equation}
and the condition $x(t_0)=x_0$, then
$G$ is surjective (onto).
\end{lemma}

\begin{proof}
Let $y \in M$. There exists a function $x(\cdot)$ which verifies
{\rm (\ref{ecalfal1.1})} and the condition $x(t_0)=y$.
For $t = t_0-1_{\beta}$, one obtains $x(t_0)=G (x(t_0-1_{\beta}))$, hence
$G (x(t_0-1_{\beta}))=y$.
Because $y$ is arbitrary, it follows that the function $G$ is surjective.
\end{proof}

\begin{proposition}
\label{alfa.p2}
We consider the functions $G_{\alpha} \colon M \to M$,
for $\alpha \in \{ 1,2 \ldots, m \}$.

\noindent
$a)$ Let $t_0 \in \mathbb{Z}^m$ and
$\alpha_0 \in \{ 1,2, \ldots, m \}$, fixed.
If for any $x_0 \in  M$, there exists at least one function
$x \colon \big\{ t \in \mathbb{Z}^m \, \big|\,
t \geq t_0-1_{\alpha_0} \big\} \to M$,
which verifies
\begin{equation}
\label{ecalfap2.1}
x(t+1_\alpha) = G_\alpha (x(t)),
\end{equation}
\begin{equation*}
\forall t \geq t_0-1_{\alpha_0},\,\,
\forall \alpha \in \{1,2, \ldots, m\},
\end{equation*}
and the condition $x(t_0)=x_0$,
then $G_{\alpha_0}$ is surjective and
\begin{equation}
\label{ecalfap2.3}
G_{\alpha}\circ G_{\beta}
=
G_{\beta} \circ G_{\alpha},
\quad
\forall \alpha, \beta \in \{1,2, \ldots, m\}.
\end{equation}

\noindent
$b)$ Suppose that, for any $\alpha\in \{1,2, \ldots, m\}$, the functions
$G_{\alpha}$ are surjective and that, for any
$\alpha, \beta \in \{1,2, \ldots, m\}$,
the relations {\rm (\ref{ecalfap2.3})} are satisfied.

Let $(t_0,x_0) \in \mathbb{Z}^m \times M$
and $s \in \mathbb{Z}^m$, $s \leq t_0$.

If for $a \in M$, we have $G_1^{(t_0^1-s^1)} \circ G_2^{(t_0^2-s^2)}  \circ
\ldots \circ G_m^{(t_0^m-s^m)}(a)=x_0$, then the function
\begin{equation*}
x \colon \big\{ t \in \mathbb{Z}^m \, \big|\,
t \geq s \big\} \to M,
\end{equation*}
\begin{equation}
\label{ecalfap2.5}
x(t)=
G_1^{(t^1-s^1)} \circ G_2^{(t^2-s^2)}  \circ
\ldots \circ G_m^{(t^m-s^m)}(a),
\quad
\forall t \geq s,
\end{equation}
verifies the recurrence {\rm (\ref{ecalfap2.1})},
$
\forall t \geq s,\,\,
\forall \alpha \in \{1,2, \ldots, m\},
$
and $x(t_0)=x_0$.

\noindent
$c)$ Suppose that, for any
$\alpha\in \{1,2, \ldots, m\}$, the functions
$G_{\alpha}$ are surjective and that, for any
$\alpha, \beta \in \{1,2, \ldots, m\}$,
the relations {\rm (\ref{ecalfap2.3})} are satisfied.

Then, for any $(t_0,x_0) \in \mathbb{Z}^m \times M$,
there exists at least one function
$x \colon  \mathbb{Z}^m  \to M$
which verifies the recurrence {\rm (\ref{ecalfap2.1})},
$\forall t \in \mathbb{Z}^m$,\,
$\forall \alpha \in \{1,2, \ldots, m\}$,
and the condition $x(t_0)=x_0$.
\end{proposition}

\begin{proof}
$a)$ The surjectivity of $G_{\alpha_0}$ follows from Lemma \ref{alfa.l1}.
The relations {\rm (\ref{ecalfap2.3})} are obtained from
Theorem {\rm \ref{alfa.t1}}, $a)$,
considering the restriction of $x(\cdot)$ to the set
$\big\{ t \in \mathbb{Z}^m \, \big|\, t \geq t_0 \big\}$.

$b)$ We observe that the function
$G_1^{(t_0^1-s^1)} \circ G_2^{(t_0^2-s^2)}  \circ
\ldots \circ G_m^{(t_0^m-s^m)}$ is surjective,
since $t_0^{\alpha}-s^{\alpha} \geq 0$, $\forall \alpha$,
and $G_\alpha$ are surjective. Consequently, there exists $a \in M$ such that
$G_1^{(t_0^1-s^1)} \circ G_2^{(t_0^2-s^2)}  \circ
\ldots \circ G_m^{(t_0^m-s^m)}(a)=x_0$.

From the Theorem {\rm \ref{alfa.t1}}, $b)$, it follows the function defined by the formula
{\rm (\ref{ecalfap2.5})} is the unique function which verifies the recurrence
{\rm (\ref{ecalfap2.1})}, $\forall t \geq s$, $\forall \alpha$,
and the condition $x(s)=a$. For $t=t_0$, we have
$$
x(t_0)=
G_1^{(t_0^1-s^1)} \circ G_2^{(t_0^2-s^2)}  \circ
\ldots \circ G_m^{(t_0^m-s^m)}(a)=x_0.
$$

$c)$ Let $G=G_1 \circ G_2  \circ \ldots \circ G_m$.
Since the functions $G_\alpha$ are surjective, it follows that the function
$G$ is surjective. Hence, there exists a function $H \colon M \to M$
such that $G \circ H ={\rm Id}_M$ (right inverse).

For $n \in \mathbb{N}$, we denote
$P_n=\big\{ t \in \mathbb{Z}^m \, \big|\,
t \geq t_0-n \cdot {\bf 1}  \big\}$; let
$a_n=H^{(n)}(x_0)$.
We observe that $G(a_{n+1})=a_n$ and
$G^{(n)}(a_n)=x_0$, $\forall n \in \mathbb{N}$.

For each $n\in \mathbb{N}$, we consider the function
$y_n \colon P_n \to M$, defined by
\begin{equation*}
y_n(t)=
G_1^{(t^1-t_0^1+n)} \circ G_2^{(t^2-t_0^2+n)}  \circ
\ldots \circ G_m^{(t^m-t_0^m+n)}(a_n),
\quad
\forall t \geq t_0-n \cdot {\bf 1}.
\end{equation*}

Because $G^{(n)}(a_n)=x_0$, i.e.,
$G_1^{(n)} \circ G_2^{(n)}  \circ
\ldots \circ G_m^{(n)}(a_n)=x_0$,
according the step $b)$, it follows that the function
$y_n$ verifies the recurrence
{\rm (\ref{ecalfap2.1})},
$\forall t \in P_n$, $\forall \alpha$
and the condition $y_n(t_0)=x_0$.

We remark that $P_n \subseteq P_{n+1}$.
For $t \in P_n$, we have
\begin{equation*}
y_{n+1}(t)=
G_1^{(t^1-t_0^1+n+1)} \circ G_2^{(t^2-t_0^2+n+1)}  \circ
\ldots \circ G_m^{(t^m-t_0^m+n+1)}(a_{n+1})
\end{equation*}
\begin{equation*}
=
G_1^{(t^1-t_0^1+n)} \circ G_2^{(t^2-t_0^2+n)}  \circ
\ldots \circ G_m^{(t^m-t_0^m+n)}\big( G(a_{n+1}) \big)
\end{equation*}
\begin{equation*}
=
G_1^{(t^1-t_0^1+n)} \circ G_2^{(t^2-t_0^2+n)}  \circ
\ldots \circ G_m^{(t^m-t_0^m+n)}( a_n )
=y_n(t).
\end{equation*}
We showed that $y_{n+1}(t)=y_n(t)$,
$\forall t \in P_n$.
Inductively, one deduces that, for any $q \in \mathbb{N}$, we have
$y_{n+q}(t)=y_n(t)$,
$\forall t \in P_n$.
Consequently, $y_n(t)=y_k(t)$,
$\displaystyle \forall t\in P_{\min \{n,k \}}$.

Let us define the function
$x \colon  \mathbb{Z}^m  \to M$.

Let $t \in \mathbb{Z}^m$.
Since
$\mathbb{Z}^m =\displaystyle \bigcup_{n \in \mathbb{N}}P_n$,
there exists $n\in \mathbb{N}$, such that
$t \in P_n$. The value of the function $x$ at $t$
will be $x(t)=y_n(t)$.

The function $x(\cdot)$ is well defined
since if $t \in P_n$ and $t\in P_k$, we have showed that
$y_n(t)=y_k(t)$.

If $t \in P_n$, then $t+1_{\alpha} \in P_n$.
We have $x(t+1_{\alpha})=y_n(t+1_{\alpha})=G_{\alpha}(y_n(t))
=G_{\alpha}(x(t))$ and $x(t_0)=y_n(t_0)=x_0$.
\end{proof}

\begin{proposition}
\label{alfa.p3}
Suppose that, for the functions $G_{\alpha} \colon M \to M$,
the relations {\rm (\ref{ecalfap2.3})} are satisfied.

Let $t_0 \in \mathbb{Z}^m$ and $\alpha_0 \in \{ 1,2, \ldots, m \}$, fixed.
If, for any $x_0 \in  M$, there exists at most one function
$x \colon \big\{ t \in \mathbb{Z}^m \, \big|\,
t \geq t_0-1_{\alpha_0} \big\} \to M$,
which verifies
\begin{equation}
\label{ecalfap3.1}
x(t+1_\alpha) = G_\alpha (x(t)),
\quad
\forall t \geq t_0-1_{\alpha_0},\,\,
\forall \alpha \in \{1,2, \ldots, m\},
\end{equation}
and the condition $x(t_0)=x_0$,
then $G_{\alpha_0}$ is injective (one-to-one).
\end{proposition}

\begin{proof}
Let $p, q \in M$ such that $G_{\alpha_0}(p)=G_{\alpha_0}(q)$.

We select $x_0=G_{\alpha_0}(p)=G_{\alpha_0}(q)$.

The functions $$x, y \colon \big\{ t \in \mathbb{Z}^m \, \big|\,
t \geq t_0-1_{\alpha_0} \big\} \to M,$$
\begin{equation}
\label{ecalfap3.3}
x(t)=
G_1^{(t^1-t_0^1)}
\circ
\ldots
\circ
G_{\alpha_0}^{(t^{\alpha_0}-t_0^{\alpha_0}+1)}
\circ
\ldots \circ G_m^{(t^m-t_0^m)}(p),
\quad
\forall
t \geq t_0-1_{\alpha_0},
\end{equation}
\begin{equation}
\label{ecalfap3.4}
y(t)=
G_1^{(t^1-t_0^1)}
\circ
\ldots
\circ
G_{\alpha_0}^{(t^{\alpha_0}-t_0^{\alpha_0}+1)}
\circ
\ldots \circ G_m^{(t^m-t_0^m)}(q),
\quad
\forall
t \geq t_0-1_{\alpha_0},
\end{equation}
are well defined (since $t^{\alpha_0}-t_0^{\alpha_0}+1\geq 0$),
verifies the relations {\rm (\ref{ecalfap3.1})} and
$x(t_0)=G_{\alpha_0}(p)=x_0$,
$y(t_0)=G_{\alpha_0}(q)=x_0$.
It follows that $x(t)=y(t)$,
$\forall t \geq t_0-1_{\alpha_0}$.
For $t = t_0-1_{\alpha_0}$, we obtain
$x(t_0-1_{\alpha_0})=y(t_0-1_{\alpha_0})$,
relation which is equivalent to $p=q$
(according {\rm (\ref{ecalfap3.3})}, {\rm (\ref{ecalfap3.4})}).
Hence, the function $G_{\alpha_0}$ is injective.
\end{proof}

If $G \colon M \to M$ is a bijective function,
we denote $G^{(-k)}=\big(G^{-1}\big)^{(k)}$,
for $k \in \mathbb{N}$; we have
$G^{(-k)}=\big(G^{(k)}\big)^{-1}$.

\begin{proposition}
\label{alfa.p4}
Suppose that the functions $G_{\alpha} \colon M \to M$ are bijective
and the relations {\rm (\ref{ecalfap2.3})} hold. Then, for any
$(t_0,x_0) \in \mathbb{Z}^m \times M$,
there exists a unique solution
$x \colon  \mathbb{Z}^m \to M$,
of the recurrence equation
\begin{equation}
\label{ecalfap4.1}
x(t+1_\alpha) = G_\alpha (x(t)),
\quad
\forall t \in \mathbb{Z}^m,\,\,
\forall \alpha \in \{1,2, \ldots, m\},
\end{equation}
with the condition $x(t_0)=x_0$.
The function $x$ is defined by the relation
\begin{equation}
\label{ecalfap4.2}
x(t)=
G_1^{(t^1-t_0^1)} \circ G_2^{(t^2-t_0^2)}  \circ
\ldots \circ G_m^{(t^m-t_0^m)}(x_0)
\qquad
(\forall t \in \mathbb{Z}^m).
\end{equation}
\end{proposition}

\begin{proof} The existence follows from the
Proposition \ref{alfa.p2}, $c)$.

Let $x \colon  \mathbb{Z}^m \to M$ be a solution
of the recurrence {\rm (\ref{ecalfap4.1})}, with
$x(t_0)=x_0$. For proving the uniqueness, it is sufficient to show that
$x(t)$ verifies the relation {\rm (\ref{ecalfap4.2})},
$\forall t \in \mathbb{Z}^m$.

Let $s\leq t_0$. We apply the Theorem \ref{alfa.t1} for the restriction
of $x$ to the set $\big\{ t \in \mathbb{Z}^m \, \big|\,
t \geq s \big\}$. It follows
\begin{equation*}
x(t)=
G_1^{(t^1-s^1)} \circ G_2^{(t^2-s^2)}  \circ
\ldots \circ G_m^{(t^m-s^m)}(x(s)),
\quad
\forall t \geq s.
\end{equation*}
For $t=t_0$, we obtain
$x_0=
G_1^{(t_0^1-s^1)} \circ G_2^{(t_0^2-s^2)}  \circ
\ldots \circ G_m^{(t_0^m-s^m)}(x(s))$.
Since the functions $G_{\alpha}$ are bijective, it follows
$$
x(s)=G_1^{(s^1-t_0^1)} \circ G_2^{(s^2-t_0^2)}  \circ
\ldots \circ G_m^{(s^m-t_0^m)}(x_0).
$$

Consequently, for any $t\geq s$, we have
\begin{equation*}
x(t)=
G_1^{(t^1-s^1)} \circ G_2^{(t^2-s^2)}  \circ
\ldots \circ G_m^{(t^m-s^m)}(x(s))
=
\end{equation*}
\begin{equation*}
=
G_1^{(t^1-s^1)} \circ G_2^{(t^2-s^2)}  \circ
\ldots \circ G_m^{(t^m-s^m)}  \circ
G_1^{(s^1-t_0^1)} \circ G_2^{(s^2-t_0^2)}  \circ
\ldots \circ G_m^{(s^m-t_0^m)}(x_0)
\end{equation*}
\begin{equation*}
=
G_1^{(t^1-t_0^1)} \circ G_2^{(t^2-t_0^2)}  \circ
\ldots \circ G_m^{(t^m-t_0^m)}(x_0).
\end{equation*}

We have showed that, for any
$s \leq t_0$ and any $t \geq s$, the sequence
$x(t)$ verifies the relation {\rm (\ref{ecalfap4.2})}.
Since
$$\displaystyle
\bigcup_{s \in  \mathbb{Z}^m,\, s \leq t_0}
\big\{ t \in \mathbb{Z}^m \, \big|\,
t \geq s \big\}
=\mathbb{Z}^m,$$ it follows that the relation {\rm (\ref{ecalfap4.2})}
holds for any $t \in \mathbb{Z}^m$.
\end{proof}

\begin{theorem}
\label{alfa.t2}
Let $M$ be a nonvoid set. For each
$\alpha \in \{ 1,2, \ldots, m \}$, we consider the function
$G_{\alpha} \colon M \to M$ and we associate the recurrence equation
\begin{equation}
\label{ecalfat2.1}
x(t+1_\alpha) = G_\alpha (x(t)),
\quad
\forall \alpha \in \{1,2, \ldots, m\}.
\end{equation}

The following statements are equivalent:

\noindent
$i)$ For any
$\alpha\in \{1,2, \ldots, m\}$, the functions
$G_{\alpha}$ are bijective and
\begin{equation}
\label{ecalfat2.2}
G_{\alpha}\circ G_{\beta}
=
G_{\beta} \circ G_{\alpha},
\quad
\forall \alpha, \beta \in \{1,2, \ldots, m\}.
\end{equation}

\noindent
$ii)$ There exists $t_0 \in \mathbb{Z}^m$ such that
$\forall \alpha_0 \in \{ 1,2, \ldots, m \}$,
$\forall x_0 \in  M$,
there exists a unique function
$x \colon \big\{ t \in \mathbb{Z}^m \, \big|\,
t \geq t_0-1_{\alpha_0} \big\} \to M$,
which, for any
$t \geq t_0-1_{\alpha_0}$, verifies the relations
{\rm (\ref{ecalfat2.1})},
and the condition $x(t_0)=x_0$.

\vspace{0.1 cm}
\noindent
$iii)$ There exist $t_0, t_1 \in \mathbb{Z}^m$,
with $t_1^{\alpha} < t_0^{\alpha}$, $\forall \alpha$,
such that, for each $x_0 \in  M$, there exists a unique function
$x \colon \big\{ t \in \mathbb{Z}^m \, \big|\,
t \geq t_1 \big\} \to M$,
which, for each $t \geq t_1$ verifies the relations {\rm (\ref{ecalfat2.1})},
and also the condition $x(t_0)=x_0$.

\vspace{0.1 cm}
\noindent
$iv)$ For each $t_0, t_1 \in \mathbb{Z}^m$,
with $t_1 \leq t_0$, and for any $x_0 \in  M$,
there exists a unique function
$x \colon \big\{ t \in \mathbb{Z}^m \, \big|\,
t \geq t_1 \big\} \to M$,
which, for any $t \geq t_1$ verifies the relations
{\rm (\ref{ecalfat2.1})},
and also the condition $x(t_0)=x_0$.

\vspace{0.1 cm}
\noindent
$v)$ There exists $t_0\in \mathbb{Z}^m$, such that,
for any $x_0 \in  M$, there exist a unique function
$x \colon \mathbb{Z}^m \to M$,
which, for any
$t \in \mathbb{Z}^m $ verifies the relations
{\rm (\ref{ecalfat2.1})},
and $x(t_0)=x_0$.

\vspace{0.1 cm}
\noindent
$vi)$ For each pair
$(t_0,x_0) \in \mathbb{Z}^m \times M$,
there exists a unique function $x \colon \mathbb{Z}^m \to M$,
which, for any $t \in \mathbb{Z}^m $ verifies the relations
{\rm (\ref{ecalfat2.1})}, and $x(t_0)=x_0$.
\end{theorem}

\begin{proof}
$ii)\Longrightarrow i)$: The relations {\rm (\ref{ecalfat2.2})}
and the surjectivity of functions $G_{\alpha}$
follow from the Proposition \ref{alfa.p2}, $a)$, and the
injectivity of the functions $G_{\alpha}$
follow from the Proposition \ref{alfa.p3}.

$i)\Longrightarrow vi)$: It follows from the Proposition \ref{alfa.p4}.

$vi)\Longrightarrow iv)$: Considering the restrictions of the functions $x$
to $\big\{ t \in \mathbb{Z}^m \, \big|\, t \geq t_0-1_{\alpha_0} \big\}$ (for each $\alpha_0$),
from the Proposition \ref{alfa.p2}, $a)$, it follows that the relations {\rm (\ref{ecalfat2.2})}
hold and that the functions $G_{\alpha}$ are surjective, $\forall \alpha$.

Let $t_0, t_1 \in \mathbb{Z}^m$, with $t_1 \leq t_0$, and $x_0 \in  M$.
There exists a unique function $\tilde{x} \colon \mathbb{Z}^m \to M$
such that $\tilde{x}(t_0)=x_0$ and the relations {\rm (\ref{ecalfat2.1})} are true,
$\forall t \in \mathbb{Z}^m $.

To prove the existence, it is sufficient to select $x$ as the restriction of
$\tilde{x}$ to $ \big\{ t \in \mathbb{Z}^m \, \big|\, t \geq t_1 \big\}$.

Let $y \colon \big\{ t \in \mathbb{Z}^m \, \big|\,
t \geq t_1 \big\} \to M$, be a function such that
$y(t_0)=x_0$ and for which the relations {\rm (\ref{ecalfat2.1})} hold,
$\forall t \geq t_1$. We shall prove that the functions $x$ and $y$ are equal.

From the Proposition \ref{alfa.p2}, $c)$, there exists
$\tilde{y} \colon  \mathbb{Z}^m  \to M$
such that $\tilde{y}(t_1)=y(t_1)$ and for which the relations
{\rm (\ref{ecalfat2.1})} hold,
$\forall t \in \mathbb{Z}^m $.
From the Theorem \ref{alfa.t1}, it follows that $y$
and its restriction $\tilde{y}$
to $\big\{ t \in \mathbb{Z}^m \, \big|\,
t \geq t_1 \big\}$ coincide. Since
$t_0 \geq t_1$, we have
$\tilde{y}(t_0)=y(t_0)=x_0$.
It follows that the functions $\tilde{x}$ and $\tilde{y}$ coincide.
Consequently, for each $t \geq t_1$, we have
\begin{equation*}
y(t)=\tilde{y}(t)=\tilde{x}(t)=x(t).
\end{equation*}

$iv)\Longrightarrow ii)$ is an obvious implication.

We have proved that the statements $i)$, $ii)$, $iv)$, $vi)$ are equivalent.

$i)\Longrightarrow iii)$: We have $i) \Longleftrightarrow iv)$,
and $iv)\Longrightarrow iii)$ is obvious.

$iii)\Longrightarrow i)$: For each $\alpha$,
we have $t_1^{\alpha} < t_0^{\alpha}$, i.e.,
$t_0^{\alpha}-1 \geq t_1^{\alpha}$.
Hence $t_0 - 1_{\alpha} \geq t_1$, $\forall \alpha$.
Considering the restrictions of the functions $x$
to $\big\{ t \in \mathbb{Z}^m \, \big|\,
t \geq t_0-1_{\alpha_0} \big\}$ (for each $\alpha_0$),
from the Proposition \ref{alfa.p2}, $a)$,
it follows that the relations {\rm (\ref{ecalfat2.2})}
are true and the functions $G_{\alpha}$ are surjective,
$\forall \alpha$.

Let $\alpha_0 \in \{1,2, \ldots, m\}$.
We shall prove that $G_{\alpha_0}$ is injective.

Let $p, q \in M$ such that $G_{\alpha_0}(p)=G_{\alpha_0}(q)$.
According the Proposition \ref{alfa.p2}, $c)$,
there exist the functions
$y, z \colon  \mathbb{Z}^m  \to M$
for which the relations {\rm (\ref{ecalfat2.1})} hold,
$\forall t \in \mathbb{Z}^m$,
and
$y(t_0-1_{\alpha_0})=p$,
$z(t_0-1_{\alpha_0})=q$.
Let $x_0=G_{\alpha_0}(p)=G_{\alpha_0}(q)$.
\begin{equation*}
y(t_0)=G_{\alpha_0}(y(t_0-1_{\alpha_0}))
=G_{\alpha_0}(p)=x_0,
\end{equation*}
\begin{equation*}
z(t_0)=G_{\alpha_0}(z(t_0-1_{\alpha_0}))
=G_{\alpha_0}(q)=x_0.
\end{equation*}
Applying the uniqueness property for the restrictions of the functions
$y$ and $z$ to the set
$\big\{ t \in \mathbb{Z}^m \, \big|\,
t \geq t_1 \big\}$, we obtain
$y(t)=z(t)$, $\forall t \geq t_1$.

Since $t_0 - 1_{\alpha_0} \geq t_1$, it follows
$y(t_0-1_{\alpha_0})=z(t_0-1_{\alpha_0})$, i.e.,
$p=q$.

$i)\Longrightarrow v)$: We have $i) \Longleftrightarrow vi)$,
and $vi)\Longrightarrow v)$ is obvious.

$v)\Longrightarrow i)$:
For each $\alpha$, we have $t_1^{\alpha} < t_0^{\alpha}$, i.e.,
$t_0^{\alpha}-1 \geq t_1^{\alpha}$.
Hence $t_0 - 1_{\alpha} \geq t_1$, $\forall \alpha$.
Considering the restrictions of the functions $x$
to $\big\{ t \in \mathbb{Z}^m \, \big|\,
t \geq t_0-1_{\alpha_0} \big\}$ (for each $\alpha_0$),
by the Proposition \ref{alfa.p2}, $a)$,
it follows that the relations {\rm (\ref{ecalfat2.2})}
hold and that the functions $G_{\alpha}$ are surjective,
$\forall \alpha$.

Let $\alpha_0 \in \{1,2, \ldots, m\}$.
We shall prove that $G_{\alpha_0}$ is injective.

Let $p, q \in M$ such that $G_{\alpha_0}(p)=G_{\alpha_0}(q)$.
According the Proposition \ref{alfa.p2}, $c)$, there exist the functions
$y, z \colon  \mathbb{Z}^m  \to M$ for which the relations
{\rm (\ref{ecalfat2.1})} are true, $\forall t \in \mathbb{Z}^m$,
and
$y(t_0-1_{\alpha_0})=p$,
$z(t_0-1_{\alpha_0})=q$.
Let $x_0=G_{\alpha_0}(p)=G_{\alpha_0}(q)$.
\begin{equation*}
y(t_0)=G_{\alpha_0}(y(t_0-1_{\alpha_0}))
=G_{\alpha_0}(p)=x_0,
\end{equation*}
\begin{equation*}
z(t_0)=G_{\alpha_0}(z(t_0-1_{\alpha_0}))
=G_{\alpha_0}(q)=x_0.
\end{equation*}

From uniqueness, we obtain
$y(t)=z(t)$, $\forall t \in \mathbb{Z}^m $;
for $t=t_0 - 1_{\alpha_0}$ it follows
$y(t_0-1_{\alpha_0})=z(t_0-1_{\alpha_0})$, i.e.,
$p=q$.
\end{proof}

\begin{remark}
\label{alfa.o1}
We consider the bijective functions
$G_{\alpha} \colon M \to M$, for which
the relations {\rm (\ref{ecalfat2.2})} holds.
Let $t_0 \in \mathbb{Z}^m$
and let $\tilde{x} \colon
\big\{ t \in \mathbb{Z}^m \, \big|\,
t \geq t_0 \big\} \to M$ be a solution of the recurrence {\rm (\ref{ecalfat2.1})}.
From Theorem {\rm \ref{alfa.t2}}, it follows that there exists a unique function
$x \colon \mathbb{Z}^m \to M$,
solution of the recurrence {\rm (\ref{ecalfat2.1})}, such that
$x(t)=\tilde{x}(t)$, $\forall t\geq t_0$,
i.e., $x$ is an extension of $\tilde{x}$.
According Proposition {\rm \ref{alfa.p4}},
this extension is defined by the formula {\rm (\ref{ecalfap4.2})},
for $x_0=\tilde{x}(t_0)$.
\end{remark}

\section{Non-autonomous discrete multitime \\multiple recurrence}

\subsection{Existence and uniqueness Theorem}

Let $t_1 \in \mathbb{Z}^m$. Consider the functions
$F_\alpha \colon \big\{ t \in \mathbb{Z}^m \, \big|\,
t \geq t_1 \big\} \times M \to M$,
$\alpha \in \{ 1,2, \ldots, m \}$,
which define the recurrence equation
\begin{equation}
\label{ec.alfa1}
x(t+1_\alpha) = F_\alpha (t,x(t)),
\quad
\forall \alpha \in \{1,2, \ldots, m\}.
\end{equation}

Let $\widetilde{M}=
\big\{ s \in \mathbb{Z}^m \, \big|\,
s \geq t_1 \big\} \times M$ and let
$G_\alpha \colon \widetilde{M} \to \widetilde{M}$,
\begin{equation*}
G_\alpha (s,x)=
\big( s+1_\alpha, F_\alpha (s,x) \big),
\quad
\forall (s,x) \in \widetilde{M}.
\end{equation*}
The functions $G_\alpha$
define the recurrence
\begin{equation}
\label{ec.alfa2}
\big(  s(t+1_\alpha),x(t+1_\alpha)  \big)
=
\big( s(t)+1_\alpha, F_\alpha (s(t),x(t)) \big),
\quad
\forall \alpha \in \{1,2, \ldots, m\},
\end{equation}
which is equivalent to
\begin{equation}
\label{ec.alfa3}
  \begin{cases}
    x(t+1_\alpha)=F_\alpha (s(t),x(t)) \\
    s(t+1_\alpha)=s(t)+1_\alpha
  \end{cases}
\!\!\!\!\!\!,
\quad
\forall \alpha \in \{1,2, \ldots, m\}.
\end{equation}
The unknown function is
$\big(  s(\cdot), x(\cdot)  \big)$.
Denoting $y=(s,x)$, the recurrence
(\ref{ec.alfa2}) can be rewritten in the form
\begin{equation}
\label{ec.alfa4}
y(t+1_\alpha) = G_\alpha (y(t)),
\quad
\forall \alpha \in \{1,2, \ldots, m\},
\end{equation}
with the unknown function
$y(\cdot)=\big(  s(\cdot), x(\cdot)  \big)$.

\begin{lemma}
\label{alfa.l2}
$a)$ Let $t_0, t_1, s_0 \in \mathbb{Z}^m$,
with $t_0\geq t_1$.

Then the function
$s \colon \big\{ t \in \mathbb{Z}^m \, \big|\,
t \geq t_1 \big\} \to \mathbb{Z}^m$ verifies, for each $t \geq t_1$,
\begin{equation}
\label{ecalfal2.1}
s(t+1_\alpha)=s(t)+1_\alpha,
\quad
\forall \alpha \in \{1,2, \ldots, m\},
\end{equation}
\begin{equation*}
\mbox{and the condition}\,\,\,
s(t_0)=s_0,
\end{equation*}
if and only if
$s(t)=t-t_0+s_0$, $\forall t \geq t_1$.

$b)$ Let $t_0,  s_0 \in \mathbb{Z}^m$. The function
$s \colon \mathbb{Z}^m  \to \mathbb{Z}^m$ verifies,
for each $t \in \mathbb{Z}^m$, the relations {\rm (\ref{ecalfal2.1})}
and the condition $s(t_0)=s_0$ if and only if
$s(t)=t-t_0+s_0$, $\forall t  \in \mathbb{Z}^m$.
\end{lemma}

\begin{proof}
Let $\tilde{s} \colon \mathbb{Z}^m  \to \mathbb{Z}^m$,
$\tilde{s}(t)=t-t_0+s_0$, $\forall t  \in \mathbb{Z}^m$.
One observes immediately that $\tilde{s}$ verifies,
for any $t \in \mathbb{Z}^m$, the relations {\rm (\ref{ecalfal2.1})}
and $\tilde{s}(t_0)=s_0$.

For each $\alpha$, we consider the function
\begin{equation*}
\widetilde{G}_\alpha \colon \mathbb{Z}^m \to \mathbb{Z}^m,
\quad
\widetilde{G}_\alpha(s)=s+1_\alpha,
\,\,
\forall s \in \mathbb{Z}^m.
\end{equation*}

The relations {\rm (\ref{ecalfal2.1})} are equivalent to
$s(t+1_\alpha) = \widetilde{G}_\alpha (s(t))$,\,
$\forall \alpha \in \{1,2, \ldots, m\}$.

One observes that
$\widetilde{G}_\alpha\circ \widetilde{G}_\beta(s)
=
\widetilde{G}_\beta\circ \widetilde{G}_\alpha(s)
=s+1_\alpha+1_\beta$,\,
$\forall s \in \mathbb{Z}^m$.

For any $\alpha$, the function
$\widetilde{G}_\alpha$ is bijective. Its
inverse is
$(\widetilde{G}_\alpha)^{-1}(s)=s-1_\alpha$,
$\forall s \in \mathbb{Z}^m$.

According the Theorem \ref{alfa.t2}, $iv)$,
there exists a unique function
$s \colon \big\{ t \in \mathbb{Z}^m \, \big|\,
t \geq t_1 \big\} \to \mathbb{Z}^m$
which verifies the recurrence (\ref{ecalfal2.1}), $\forall t\geq t_1$, and
the condition $s(t_0)=t_0$. By uniqueness, it follows that $s$
coincides with the restriction of the function $\tilde{s}$ to the set
$\big\{ t \in \mathbb{Z}^m \, \big|\,
t \geq t_1 \big\}$;
hence $s(t)=t-t_0+s_0$, $\forall t\geq t_1$.

According the Theorem \ref{alfa.t2}, $vi)$, it follows that there exists a unique function
$\sigma \colon  \mathbb{Z}^m \to \mathbb{Z}^m$
which verifies the recurrence
(\ref{ecalfal2.1}), $\forall t \in \mathbb{Z}^m$, and the condition
$\sigma(t_0)=t_0$.
From uniqueness, it follows that $\sigma=\tilde{s}$;
hence $\sigma(t)=t-t_0+s_0$, $\forall t \in \mathbb{Z}^m$.
\end{proof}


\begin{proposition}
\label{alfa.p5}
Suppose that we are in the foregoing conditions.

\noindent
$a)$ For $\alpha, \beta \in \{1,2, \ldots, m\}$, we have
\begin{equation}
\label{ecalfap5.1}
F_\alpha(t+1_\beta,F_\beta (t,x))
=
F_\beta(t+1_\alpha,F_\alpha (t,x)),
\quad
\quad
\forall t \geq t_1,\,\,
\forall x \in M
\end{equation}
if and only if
$G_{\alpha}\circ G_{\beta}
=
G_{\beta} \circ G_{\alpha}$.

\noindent
$b)$ Let $t_0 \in \mathbb{Z}^m$, $t_0 \geq t_1$ and $x_0 \in M$.

If $x \colon \big\{ t \in \mathbb{Z}^m \, \big|\,
t \geq t_1 \big\} \to M$ verifies the recurrence
{\rm (\ref{ec.alfa1})}, $\forall t \geq t_1$,
and the condition $x(t_0)=x_0$, then the function
\begin{equation*}
y \colon \big\{ t \in \mathbb{Z}^m \, \big|\,
t \geq t_1 \big\} \to \widetilde{M},
\quad
y(t)=(t, x(t)),
\quad
\forall t\geq t_1,
\end{equation*}
verifies the recurrence
{\rm (\ref{ec.alfa4})}, $\forall t \geq t_1$,
and the condition $y(t_0)=(t_0,x_0)$.

Conversely, if
$y(\cdot)= \big( s(\cdot), x(\cdot) \big)
\colon \big\{ t \in \mathbb{Z}^m \, \big|\,
t \geq t_1 \big\} \to \widetilde{M}$
verifies the recurrence
{\rm (\ref{ec.alfa4})}, $\forall t \geq t_1$,
and the condition $y(t_0)=(t_0,x_0)$, then
$s(t)=t, \quad \forall t \geq t_1$
and $x(\cdot)$ verifies the recurrence
{\rm (\ref{ec.alfa1})}, $\forall t \geq t_1$,
and the condition $x(t_0)=x_0$.
\end{proposition}

\begin{proof}
$a)$ For any $(s,x) \in \widetilde{M}$, we have
\begin{equation*}
G_{\alpha}\circ G_{\beta} (s,x)=
G_{\beta} \circ G_{\alpha}(s,x)
\Longleftrightarrow
\end{equation*}
\begin{equation*}
\Longleftrightarrow
\big( s+1_\beta+1_\alpha, F_\alpha(s+1_\beta,F_\beta (s,x)) \big)
=
\big( s+1_\alpha+1_\beta, F_\beta(s+1_\alpha,F_\alpha (s,x)) \big)
\end{equation*}
\begin{equation*}
\Longleftrightarrow
F_\alpha(s+1_\beta,F_\beta (s,x))
=
F_\beta(s+1_\alpha,F_\alpha (s,x))
\end{equation*}

$b)$ Let $x(\cdot)$ be a solution of the recurrence {\rm (\ref{ec.alfa1})},
with $x(t_0)=x_0$. We have to show that the function
$y(t)=(t, x(t))$ verifies the relations {\rm (\ref{ec.alfa3})};
since, for that $y(\cdot)$ we have $s(t)=t$,
the relations {\rm (\ref{ec.alfa3})} become
\begin{equation}
\label{ecalfap5.2}
  \begin{cases}
    x(t+1_\alpha)=F_\alpha (t,x(t)) \\
   t+1_\alpha=t+1_\alpha
  \end{cases}
\!\!\!\!\!\!,
\quad
\forall \alpha \in \{1,2, \ldots, m\}.
\end{equation}
The second relation in (\ref{ecalfap5.2}) is obvious,
and the first is true because $x(\cdot)$ is a solution of the recurrence {\rm (\ref{ec.alfa1})}.

The relation $y(t_0)=(t_0,x_0)$ is obvious.

Conversely, let $y(\cdot)= \big( s(\cdot), x(\cdot) \big)$
be a solution of the recurrence {\rm (\ref{ec.alfa4})},
with $y(t_0)=(t_0,x_0)$. Hence $s(\cdot)$ and $x(\cdot)$ verifies the relations
{\rm (\ref{ec.alfa3})} and the condition $s(t_0)=t_0$,
$x(t_0)=x_0$.

Since $s(t+1_\alpha)=s(t)+1_\alpha$,\,
$\forall t \geq t_1$, $\forall \alpha$,
and $s(t_0)=t_0$, from Lemma \ref{alfa.l2}
it follows that $s(t)=t$, $\forall t\geq t_1$.

Hence, the first relation in {\rm (\ref{ec.alfa3})} becomes
$x(t+1_\alpha) = F_\alpha (t,x(t))$,\,
i.e., $x(\cdot)$ is solution of the recurrence {\rm (\ref{ec.alfa1})}.
\end{proof}

The Proposition \ref{alfa.p1} presents necessary conditions for the existence
of solutions of a discrete multitime multiple recurrence.
The next Theorem shows that these conditions are also sufficient
for the existence and uniqueness of solutions.

\begin{theorem}
\label{alfa.t3}
Let $M$ be an arbitrary nonvoid set and $t_0 \in \mathbb{Z}^m$.
We consider the functions
$F_\alpha \colon \big\{ t \in \mathbb{Z}^m \, \big|\,
t \geq t_0 \big\} \times M \to M$,
$\alpha \in \{ 1,2, \ldots, m \}$,
such that
\begin{equation}
\label{alfat3.1}
F_\alpha(t+1_\beta,F_\beta (t,x))
=
F_\beta(t+1_\alpha,F_\alpha (t,x)),
\end{equation}
$$
\forall t \geq t_0,\,\,
\forall x \in M,
\quad
\forall \alpha, \beta \in \{1,2, \ldots, m\}.
$$
\noindent
Then, for any $x_0 \in  M$, there exists a unique function
$x \colon \big\{ t \in \mathbb{Z}^m \, \big|\,
t \geq t_0 \big\} \to M$ which verifies
\begin{equation}
\label{alfat3.2}
x(t+1_\alpha) = F_\alpha (t,x(t)),
\quad
\forall t\geq t_0,
\quad
\forall \alpha \in \{1,2, \ldots, m\},
\end{equation}
and the condition $x(t_0)=x_0$.
\end{theorem}

\begin{proof}
Let $\widetilde{M}=
\big\{ s \in \mathbb{Z}^m \, \big|\,
s \geq t_0 \big\} \times M$ and let
$G_\alpha \colon \widetilde{M} \to \widetilde{M}$,
\begin{equation*}
G_\alpha (s,x)=
\big( s+1_\alpha, F_\alpha (s,x) \big),
\quad
\forall (s,x) \in \widetilde{M}.
\end{equation*}

We apply the Proposition \ref{alfa.p5} (for $t_1=t_0$);
according the step $a)$, it follows that
$G_{\alpha}\circ G_{\beta}
=
G_{\beta} \circ G_{\alpha}$,
$\forall \alpha \in \{1,2, \ldots, m\}$.
From Theorem \ref{alfa.t1}, $b)$, it follows that there exists a unique function
$y(\cdot)= \big( s(\cdot), x(\cdot) \big)
\colon \big\{ t \in \mathbb{Z}^m \, \big|\,
t \geq t_0 \big\} \to \widetilde{M}$
which verifies
\begin{equation}
\label{alfat3.3}
y(t+1_\alpha) = G_\alpha (y(t)),
\quad
\forall t \geq t_0,
\quad
\forall \alpha \in \{1,2, \ldots, m\},
\end{equation}
and the condition $y(t_0)=(t_0,x_0)$.
From Proposition \ref{alfa.p5}, $b)$, it follows that $x(\cdot)$
verifies the relations {\rm (\ref{alfat3.2})}
and the condition $x(t_0)=x_0$.

Uniqueness of $x(\cdot)$: let $\tilde{x} \colon \big\{ t \in \mathbb{Z}^m \, \big|\,
t \geq t_0 \big\} \to M$ a function which verifies the relations {\rm (\ref{alfat3.2})}
and the condition $\tilde{x} (t_0)=x_0$.
From Proposition \ref{alfa.p5}, $b)$, it follows that the function
\begin{equation*}
\tilde{y} \colon \big\{ t \in \mathbb{Z}^m \, \big|\,
t \geq t_0 \big\} \to \widetilde{M},
\quad
\tilde{y}(t)=(t, \tilde{x}(t)),
\quad
\forall t\geq t_0,
\end{equation*}
verifies the relations {\rm (\ref{alfat3.3})}
and the condition $\tilde{y}(t_0)=(t_0,x_0)$.

From the uniqueness property of the solution of the recurrence {\rm (\ref{alfat3.3})}
(Theorem \ref{alfa.t1}, $b)$), it follows that the functions $y$ and $\tilde{y}$ coincide;
hence $\big( s(t), x(t) \big)=\big( t, \tilde{x}(t) \big)$,
$\forall t\geq t_0$;
we obtain $x(t)=\tilde{x}(t)$, $\forall t\geq t_0$.
\end{proof}

\begin{remark}
\label{alfa.o2}
Let $t_0, t_1, s_0 \in \mathbb{Z}^m$, $s_0 \geq t_1$ and $x_0 \in M$.

Let $\widetilde{M}=
\big\{ s \in \mathbb{Z}^m \, \big|\,
s \geq t_1 \big\} \times M$.

Suppose that $\tilde{x} \colon \big\{ t \in \mathbb{Z}^m \, \big|\,
t \geq s_0 \big\} \to M$ verifies, for any $t \geq s_0$, the recurrence {\rm (\ref{ec.alfa1})},
and the condition $\tilde{x}(s_0)=x_0$. Let $$s, x \colon \big\{ t \in \mathbb{Z}^m \, \big|\,
t \geq t_0 \big\} \to M,$$
\begin{equation*}
s(t)=t-t_0+s_0,
\quad
x(t)=\tilde{x}(t-t_0+s_0),
\quad
\forall t\geq t_0.
\end{equation*}

Easily be observed that the function
$y(\cdot)=\big(  s(\cdot), x(\cdot)  \big)
\colon \big\{ t \in \mathbb{Z}^m \, \big|\,
t \geq t_0 \big\} \to \widetilde{M}$
verifies, for each $t \geq t_0$, the recurrence {\rm (\ref{ec.alfa4})} and the condition
$y(t_0)=(s_0,x_0)$.

From here one obtains a new proof for the Proposition {\rm \ref{alfa.p1}};
Indeed, in the hypotheses of Proposition {\rm \ref{alfa.p1}}, it follows that
$\forall (s_0,x_0) \in \widetilde{M}$,
there exists a function $y(\cdot)$, which verifies, for each $t \geq t_0$,
the recurrence {\rm (\ref{ec.alfa4})} and the condition $y(t_0)=(s_0,x_0)$.
From Theorem {\rm \ref{alfa.t1}}, it follows that
$G_{\alpha}\circ G_{\beta}
=
G_{\beta} \circ G_{\alpha}$,
$\forall \alpha, \beta$, and from
Proposition {\rm \ref{alfa.p5}}
we obtain the relations {\rm (\ref{ecalfap5.1})}.
\end{remark}

\subsection{Extension theorems}

\begin{proposition}
\label{alfa.p6}
Let $\alpha_0 \in \{ 1,2, \ldots, m \}$,
$t_0 \in \mathbb{Z}^m$.

\noindent
$a)$ Let $F \colon \big\{ t \in \mathbb{Z}^m \, \big|\,
t \geq t_0-1_{\alpha_0} \big\} \times M \to M$.
If, for any $x_0 \in  M$, there exists at least one function
$x \colon \big\{ t \in \mathbb{Z}^m \, \big|\,
t \geq t_0-1_{\alpha_0} \big\} \to M$,
which verifies
\begin{equation}
\label{ecalfap6.1}
x(t+1_{\alpha_0}) = F (t,x(t)),
\quad
\forall t \geq t_0-1_{\alpha_0},
\end{equation}
and the condition $x(t_0)=x_0$, then
the function $F (t_0-1_{\alpha_0},\cdot)$ is surjective.

\noindent
$b)$ Suppose that for the functions $F_\alpha \colon \big\{ t \in \mathbb{Z}^m \, \big|\,
t \geq t_0-1_{\alpha_0} \big\} \times M \to M$,
the relations {\rm (\ref{alfat3.1})} hold,
$\forall t \geq t_0-1_{\alpha_0}$, $\forall x \in M$,
$\forall \alpha, \beta \in \{ 1,2, \ldots, m \}$.
If, for any $x_0 \in  M$, there exists at most one function
$x \colon \big\{ t \in \mathbb{Z}^m \, \big|\,
t \geq t_0-1_{\alpha_0} \big\} \to M$,
which verifies
\begin{equation}
\label{ecalfap6.2}
x(t+1_\alpha) = F_\alpha (t,x(t)),
\quad
\forall t \geq t_0-1_{\alpha_0},\,\,
\forall \alpha \in \{1,2, \ldots, m\},
\end{equation}
and the condition $x(t_0)=x_0$, then the function $F_{\alpha_0}(t_0-1_{\alpha_0},\cdot)$ is injective.
\end{proposition}

\begin{proof}
$a)$ Let $z \in M$. There exists a function $x(\cdot)$ which verifies
{\rm (\ref{ecalfap6.1})} and the condition $x(t_0)=z$.
For $t = t_0-1_{\alpha_0}$, one obtains
$z=F(t_0-1_{\alpha_0}, x(t_0-1_{\alpha_0}))$.
Since $z$ is arbitrary, it follows that
$F (t_0-1_{\alpha_0},\cdot)$ is surjective.

$b)$ Let $p, q \in M$ such that
$F_{\alpha_0}(t_0-1_{\alpha_0},p)=
F_{\alpha_0}(t_0-1_{\alpha_0},q)$.
We can apply the Theorem \ref{alfa.t3}. There exist the functions
$x, \tilde{x} \colon  \big\{ t \in \mathbb{Z}^m \, \big|\,
t \geq t_0-1_{\alpha_0} \big\} \to M$
for which the relations {\rm (\ref{ecalfap6.2})} are true, and
$x(t_0-1_{\alpha_0})=p$,
$\tilde{x}(t_0-1_{\alpha_0})=q$.

Let $x_0=F_{\alpha_0}(t_0-1_{\alpha_0},p)=
F_{\alpha_0}(t_0-1_{\alpha_0},q)$. Then
\begin{equation*}
x(t_0)=F_{\alpha_0}(t_0-1_{\alpha_0},x(t_0-1_{\alpha_0}))
=F_{\alpha_0}(t_0-1_{\alpha_0},p)=x_0,
\end{equation*}
\begin{equation*}
\tilde{x}(t_0)=
F_{\alpha_0}(t_0-1_{\alpha_0},\tilde{x}(t_0-1_{\alpha_0}))
=F_{\alpha_0}(t_0-1_{\alpha_0},q)=x_0.
\end{equation*}
It follows that the functions $x$ and $\tilde{x}$ coincide;
hence $x(t_0-1_{\alpha_0})= \tilde{x}(t_0-1_{\alpha_0})$,
i.e., $p=q$.
\end{proof}

\begin{lemma}
\label{alfa.l3}
Let $\beta \in \{ 1,2, \ldots, m \}$
and $F \colon  \mathbb{Z}^m \times M \to M$.

Let $G \colon  \mathbb{Z}^m \times M \to \mathbb{Z}^m \times M$,\,\,
$G(t,x)=\big( t+1_{\beta} ,  F (t,x) \big)$,
$\forall (t,x) \in \mathbb{Z}^m \times M$.

$a)$ The function $G$ is injective if and only if, for any $t \in  \mathbb{Z}^m$,
the function $F(t, \cdot)$ is injective.

$b)$ The function $G$ is surjective if and only if, for any $t \in  \mathbb{Z}^m$,
the function $F(t, \cdot)$ is surjective.
\end{lemma}

\begin{proof}
$a)$ Suppose that the function $G$ is injective.
Let $t \in \mathbb{Z}^m$.

If, for $x_1, x_2 \in M$, we have $F (t,x_1)=F (t,x_2)$,
then $G (t,x_1)=G (t,x_2)$. It follows $(t,x_1)=(t,x_2)$; hence $x_1=x_2$.

Conversely, let us suppose that the functions $F(t, \cdot)$ are injective.
If $G (t_1,x_1)=G (t_2,x_2)$, then
$\big( t_1+1_{\beta} ,  F (t_1,x_1) \big)=
\big( t_2+1_{\beta} ,  F (t_2,x_2) \big)$.
Hence $t_1=t_2$ and $ F (t_1,x_1)= F (t_1,x_2)$.
It follows that $x_1=x_2$. We have obtained $(t_1,x_1)=(t_2,x_2)$.

$b)$ Suppose that $G$ is surjective.
Let $t \in \mathbb{Z}^m$. If $y \in M$,
then there exists $(s, x) \in \mathbb{Z}^m \times M$,
such that $G(s,x)=(t+1_{\beta},y)$, equivalent to $\big( s+1_{\beta} ,  F (s,x) \big)
=\big(t+1_{\beta},y\big)$. It follows that $s=t$ and $F (s,x)=y$. Hence $F(t,x)=y$.

Conversely, let us suppose that the functions $F(t, \cdot)$ are surjective.

Let $(s, y) \in \mathbb{Z}^m \times M$.
There exists $x\in M$, such that $F (s-1_{\beta},x)=y$. We have
$G (s-1_{\beta},x)=
\big( s, F(s-1_{\beta},x) \big)=(s,y)$.
\end{proof}

\begin{theorem}
\label{alfa.t4}
Let $M$ be a nonvoid set. For each
$\alpha \in \{ 1,2, \ldots, m \}$, we consider the function
$F_{\alpha} \colon \mathbb{Z}^m \times M \to M$,
to which we associate the recurrence equation
\begin{equation}
\label{ecalfat4.1}
x(t+1_\alpha) = F_\alpha (t,x(t)),
\quad
\forall \alpha \in \{1,2, \ldots, m\}.
\end{equation}

The following statements are equivalent:

\noindent
$i)$ For any $\alpha\in \{1,2, \ldots, m\}$
and any $t\in \mathbb{Z}^m$, the functions $F_{\alpha}(t, \cdot)$ are bijective
and
\begin{equation}
\label{ecalfat4.2}
F_\alpha(t+1_\beta,F_\beta (t,x))
=
F_\beta(t+1_\alpha,F_\alpha (t,x)),
\end{equation}
$$
\forall (t,x) \in \mathbb{Z}^m \times M,
\,\,\,
\forall \alpha, \beta \in \{1,2, \ldots, m\}.
$$

\noindent
$ii)$ For any pair $(t_0,x_0) \in \mathbb{Z}^m \times M$,
and any index $\alpha_0 \in \{ 1,2, \ldots, m \}$,
there exists a unique function
$x \colon \big\{ t \in \mathbb{Z}^m \, \big|\,
t \geq t_0-1_{\alpha_0} \big\} \to M$,
which, for each $t \geq t_0-1_{\alpha_0}$ verifies the relations
{\rm (\ref{ecalfat4.1})},
and also the condition $x(t_0)=x_0$.

\vspace{0.1 cm}
\noindent
$iii)$ For any $t_0, t_1 \in \mathbb{Z}^m$, with $t_1 \leq t_0$,
and any $x_0 \in  M$, there exists a function
$x \colon \big\{ t \in \mathbb{Z}^m \, \big|\,
t \geq t_1 \big\} \to M$,
which, for any $t \geq t_1$, verifies the relations
{\rm (\ref{ecalfat4.1})}, and also the condition $x(t_0)=x_0$.

\vspace{0.1 cm}
\noindent
$iv)$ For any
$(t_0,x_0) \in \mathbb{Z}^m \times M$,
there exists a unique function
$x \colon \mathbb{Z}^m \to M$,
which, for any $t \in \mathbb{Z}^m $, verifies the relation
{\rm (\ref{ecalfat4.1})},
and also $x(t_0)=x_0$.
\end{theorem}

\begin{proof}
$ii)\Longrightarrow i)$: Let $t_1 \in \mathbb{Z}^m$.
For any $t_0 \geq t_1$ and any $x_0 \in  M$,
there exists a unique function
$x \colon \big\{ t \in \mathbb{Z}^m \, \big|\,
t \geq t_0-1_{\alpha_0} \big\} \to M$,
which, for any $t \geq t_0-1_{\alpha_0}$, verifies the relations {\rm (\ref{ecalfat4.1})},
and also the condition $x(t_0)=x_0$.

The restriction of the function $x(\cdot)$ to
$\big\{ t \in \mathbb{Z}^m \, \big|\, t \geq t_0 \big\}$ verifies, for any $t \geq t_0$,
the recurrence {\rm (\ref{ecalfat4.1})},
and also the condition $x(t_0)=x_0$.

From the Proposition \ref{alfa.p1} it follows that the
relations {\rm (\ref{ecalfat4.2})} hold, for any $t\geq t_1$.
Since $t_1$ is arbitrary,
we deduce that the relations {\rm (\ref{ecalfat4.2})}
are true, for any $t\in \mathbb{Z}^m $.

The surjectivity of functions $F_{\alpha}(t, \cdot)$
follows from Proposition \ref{alfa.p6}, $a)$.

The injectivity of functions $F_{\alpha}(t, \cdot)$
follows from the Proposition \ref{alfa.p6}, $b)$.

$i)\Longrightarrow iv)$: For each $\alpha \in \{ 1,2, \ldots, m \}$, we consider the function
\begin{equation*}
G_{\alpha} \colon  \mathbb{Z}^m \times M \to \mathbb{Z}^m \times M,
\quad
G_{\alpha}(t,x)=\big( t+1_{\alpha} ,  F (t,x) \big),\,\,\,
\forall (t,x) \in \mathbb{Z}^m \times M.
\end{equation*}
Similar to the proof
of Proposition \ref{alfa.p5}, it is shown that the relations {\rm (\ref{ecalfat4.2})}
are true, for any $(t,x) \in \mathbb{Z}^m \times M$
if and only if
$G_{\alpha}\circ G_{\beta}=G_{\beta} \circ G_{\alpha}$.

From Lemma \ref{alfa.l3}, we deduce that, for any $\alpha\in \{1,2, \ldots, m\}$, the function
$G_{\alpha}$ is bijective.

Let $(t_0,x_0) \in \mathbb{Z}^m \times M$. According the Theorem \ref{alfa.t2}, $vi)$,
there exists a unique function
$y(\cdot)=(s(\cdot), x(\cdot)) \colon \mathbb{Z}^m \to
\mathbb{Z}^m \times M$,
which, for any $t \in \mathbb{Z}^m $ verifies the relations
\begin{equation}
\label{ecalfat4.3}
y(t+1_\alpha) = G_\alpha (y(t)),
\quad
\forall \alpha \in \{1,2, \ldots, m\},
\end{equation}
and $y(t_0)=(t_0,x_0)$,
which are equivalent to
\begin{equation}
\label{ecalfat4.4}
  \begin{cases}
    x(t+1_\alpha)=F_\alpha (s(t),x(t)) \\
    s(t+1_\alpha)=s(t)+1_\alpha
  \end{cases}
\!\!\!\!\!\!,
\quad
\forall \alpha \in \{1,2, \ldots, m\}
\end{equation}
and $s(t_0)=t_0$, $x(t_0)=x_0$.

From Lemma \ref{alfa.l2}, we obtain
$s(t)=t$, $\forall t \in \mathbb{Z}^m$.
Replacing in the first relation of (\ref{ecalfat4.4}),
it follows that the function
$x \colon \mathbb{Z}^m \to M$ verifies the
relations {\rm (\ref{ecalfat4.1})},
$\forall t \in \mathbb{Z}^m$.

Uniqueness of $x(\cdot)$: let $\tilde{x} \colon \mathbb{Z}^m  \to M$
be a function which verifies the relations {\rm (\ref{ecalfat4.1})},
$\forall t \in \mathbb{Z}^m$,
and the condition $\tilde{x} (t_0)=x_0$.
Easily finds that the function
\begin{equation*}
\tilde{y} \colon  \mathbb{Z}^m \to \mathbb{Z}^m \times M,
\quad
\tilde{y}(t)=(t, \tilde{x}(t)),
\quad
\forall t \in \mathbb{Z}^m,
\end{equation*}
verifies the relations {\rm (\ref{ecalfat4.3})},
$\forall t \in \mathbb{Z}^m$,
and the condition $\tilde{y}(t_0)=(t_0,x_0)$.

From the uniqueness property of solutions of the recurrence {\rm (\ref{ecalfat4.3})}
(according Theorem \ref{alfa.t2}, $vi)$) it follows that the
functions $y$ and $\tilde{y}$ coincide;
hence $\big( t, x(t) \big)=\big( t, \tilde{x}(t) \big)$,
$\forall t \in \mathbb{Z}^m$; we obtain $x(t)=\tilde{x}(t)$,
$\forall t \in \mathbb{Z}^m$.

$iv)\Longrightarrow iii)$: For each
$\alpha \in \{ 1,2, \ldots, m \}$, we consider the function
$G_{\alpha}$ defined as in the proof of the
implication $i)\Longrightarrow iv)$.

Let $t_0 \in \mathbb{Z}^m$ and $(s_0,x_0) \in \mathbb{Z}^m \times M$.
We shall show that there exists a unique function
$y(\cdot)=(s(\cdot), x(\cdot)) \colon \mathbb{Z}^m \to
\mathbb{Z}^m \times M$,
which, for any $t \in \mathbb{Z}^m $, verifies the relations
{\rm (\ref{ecalfat4.3})}, and $y(t_0)=(s_0,x_0)$.

There exists a unique function
$\tilde{x} \colon \mathbb{Z}^m \to M$
which, for any $t \in \mathbb{Z}^m $ verifies the relations
{\rm (\ref{ecalfat4.1})},
and also the condition $\tilde{x}(s_0)=x_0$.

Let
$$
s \colon \mathbb{Z}^m \to \mathbb{Z}^m,
\quad
x \colon \mathbb{Z}^m \to M,
$$
\begin{equation*}
s(t)=t-t_0+s_0,
\quad
x(t)=\tilde{x}(t-t_0+s_0),
\quad
\forall t \in \mathbb{Z}^m.
\end{equation*}

Easily finds that the function
$y(\cdot)=(s(\cdot), x(\cdot)) \colon \mathbb{Z}^m \to
\mathbb{Z}^m \times M$
verifies, for any $t \in \mathbb{Z}^m$, the
recurrence {\rm (\ref{ecalfat4.3})} and
$y(t_0)=(s_0,x_0)$.

The uniqueness of $y(\cdot)$: Let $\tilde{y}(\cdot)=( \sigma (\cdot), z(\cdot))
\colon \mathbb{Z}^m \to \mathbb{Z}^m \times M$ which
verifies, for any $t \in \mathbb{Z}^m$, the recurrence {\rm (\ref{ecalfat4.3})} and the condition
$\tilde{y}(t_0)=(s_0,x_0)$.
Hence, for any $t \in \mathbb{Z}^m$, the functions $ \sigma (\cdot)$, $z(\cdot)$
verify the recurrence {\rm (\ref{ecalfat4.4})} and $\sigma (t_0)=s_0$, $z(t_0)=x_0$.
From Lemma \ref{alfa.l2}, it follows
$\sigma (t)=t-t_0+s_0=s(t)$,
$\forall t \in \mathbb{Z}^m$.

One observes immediately that the function
\begin{equation*}
\tilde{z} \colon \mathbb{Z}^m \to M,
\quad
\tilde{z}(t)=z(t+t_0-s_0),
\quad
\forall t \in \mathbb{Z}^m,
\end{equation*}
verifies the recurrence
{\rm (\ref{ecalfat4.1})}, $\forall t \in \mathbb{Z}^m $,
and $\tilde{z}(s_0)=x_0$. It follows that the functions
$\tilde{x}$ and $\tilde{z}$ coincide.
Hence, $\forall t\in \mathbb{Z}^m$, we have
$\tilde{z}(t-t_0+s_0)=\tilde{x}(t-t_0+s_0)$,
i.e., $z(t)=x(t)$.

Since $\sigma (\cdot)=s(\cdot)$ and
$z(\cdot)=x(\cdot)$, we have $\tilde{y}(\cdot)=y(\cdot)$.

Hence, for the recurrence {\rm (\ref{ecalfat4.3})}
we can apply the Theorem \ref{alfa.t2}, implication $vi)\Longrightarrow iv)$.

Let $t_0, t_1 \in \mathbb{Z}^m$,
with $t_1 \leq t_0$, and $x_0 \in  M$.

There exists a unique function
$\tilde{x} \colon \mathbb{Z}^m \to M$,
such that $\tilde{x}(t_0)=x_0$ and the relations
{\rm (\ref{ecalfat4.1})} are true,
$\forall t \in \mathbb{Z}^m $.
It is sufficient to select $x$ as being the restriction of $\tilde{x}$ to
$ \big\{ t \in \mathbb{Z}^m \, \big|\, t \geq t_1 \big\}$.

Uniqueness of the function $x(\cdot)$:
Let $z \colon \big\{ t \in \mathbb{Z}^m \, \big|\,
t \geq t_1 \big\} \to M$,
which, for any $t \geq t_1$ verifies the recurrence
{\rm (\ref{ecalfat4.1})},
and $z(t_0)=x_0$.

Let
$y, \tilde{y} \colon \big\{ t \in \mathbb{Z}^m \, \big|\,
t \geq t_1 \big\} \to \mathbb{Z}^m \times M$,
\begin{equation*}
y(t)=(t,x(t)),
\quad
\tilde{y}(t)=(t,z(t)),
\quad
\forall t \geq t_1.
\end{equation*}

One observes immediately that $y$ and $\tilde{y}$
verify the recurrence {\rm (\ref{ecalfat4.3})},
$\forall t \geq t_1$;
we have also $y(t_0)=\tilde{y}(t_0)=(t_0,x_0)$.
From Theorem \ref{alfa.t2}, it follows
that the functions $y$ and $\tilde{y}$ coincide.
We obtain $(t,x(t))=(t,z(t))$, $\forall t \geq t_1$;
hence $x(t)=z(t)$, $\forall t \geq t_1$.

$iii)\Longrightarrow ii)$ is an obvious implication.
\end{proof}


\section{Example}

Let $M$ be a nonvoid set,
$\big ( N, \cdot, e \big )$ be a monoid
and let $\varphi  \colon N \times M \to M$
be an action of the monoid $N$ on the set $M$, i.e.,
\begin{equation}
\label{act}
\varphi (ab, x)= \varphi \big( a, (b, x) \big),\,\,\,
\varphi (e, x)=x,
\quad
\forall a,b \in N,
\forall x \in M.
\end{equation}
For each $a\in N$, $x\in M$, we denote
$\varphi (a, x)=ax$ (not to be confused with the monoid operation $N$).
The relations {\rm (\ref{act})} become
\begin{equation*}
(ab) x=  a(b x) ,\,\,\,
ex=x,
\quad
\forall a,b \in N,
\forall x \in M.
\end{equation*}

\noindent
We consider $a_1,a_2, \ldots, a_m \in N$,
such that
$a_\alpha a_\beta = a_\beta a_\alpha$,
$\forall \alpha, \beta \in \{1,2, \ldots, m\}$.

For each pair $(t_0,x_0) \in \mathbb{Z}^m \times M$,
the recurrence
\begin{equation}
\label{ecalfaexm1.1}
x(t+1_\alpha) = a_\alpha x(t),
\quad
\forall \alpha \in \{1,2, \ldots, m\},
\end{equation}
with the initial condition $x(t_0)=x_0$, has unique solution
$$
x \colon \big\{ t \in \mathbb{Z}^m \, \big|\,
t \geq t_0 \big\} \to M,
$$
\begin{equation}
\label{ecalfaexm1.2}
x(t)=a_1^{(t^1-t^1_0)}a_2^{(t^2-t^2_0)}
\cdot
\ldots
\cdot
a_m^{(t^m-t^m_0)}x_0.
\end{equation}
This can be obtained by applying Theorem {\rm \ref{alfa.t1}} for the functions
$G_{\alpha} \colon M \to M$,
$G_{\alpha}(x)=a_{\alpha}x$,
$\forall x \in M$.
We have
$G_{\alpha}\circ G_{\beta}(x)
=
G_{\beta} \circ G_{\alpha}(x)=
a_{\alpha} a_{\beta} x$.
One observes that, for any $t\in \mathbb{N}^m$,
\begin{equation}
\label{ecalfaexm1.3}
G_1^{(t^1)} \circ G_2^{(t^2)}  \circ
\ldots \circ G_m^{(t^m)}(x)
=a_1^{t^1}a_2^{t^2}
\cdot
\ldots
\cdot
a_m^{t^m}x.
\end{equation}

Suppose that for any $\alpha \in \{1,2, \ldots, m\}$,
$a_{\alpha}$ is invertible;
then $G_{\alpha}$ is bijective, with the inverse
$G_{\alpha}^{-1}(x)=a_{\alpha}^{-1}x$.
We find that the formula
{\rm (\ref{ecalfaexm1.3})} is true for any $t \in \mathbb{Z}^m$.
According the Remark {\rm \ref{alfa.o1}},
there exists a unique function
$\tilde{x} \colon \mathbb{Z}^m \to M$,
solution of the recurrence {\rm (\ref{ecalfaexm1.1})},
with $\tilde{x}(t_0)=x_0$; the function
$\tilde{x}(\cdot)$ is a unique extension
of $x(\cdot)$ and it is defined
by the formula {\rm (\ref{ecalfaexm1.2})},
but for each $t\in \mathbb{Z}^m$.

\vspace{0.2 cm}
{\bf Particular cases:}

{\bf 1)} Let $(M, +)$ be a commutative monoid;
we consider the action of $M$ on himself
\begin{equation*}
\varphi  \colon M \times M \to M,
\quad
\varphi (a,x)=a+x,
\quad
\forall a,x \in M.
\end{equation*}

In this case the recurrence {\rm (\ref{ecalfaexm1.1})} becomes
\begin{equation*}
x(t+1_\alpha) = a_\alpha + x(t),
\quad
\forall \alpha \in \{1,2, \ldots, m\},
\end{equation*}
and the formula {\rm (\ref{ecalfaexm1.2})} can be written
\begin{equation*}
x(t)=(t^1-t^1_0)a_1+(t^2-t^2_0)a_2
+
\ldots
+
(t^m-t^m_0)a_m+x_0.
\end{equation*}

{\bf 2)} Let $K$ be a field.

We consider
$\big ( N, \cdot, e \big )=\big ( \mathcal{M}_n(K), \cdot, I_n \big )$,
$M=K^n=\mathcal{M}_{n,1}(K)$ and the action
\begin{equation*}
\varphi  \colon \mathcal{M}_n(K) \times K^n \to K^n,
\quad
\varphi (A,x)=Ax,
\quad
\forall A \in \mathcal{M}_n(K) , \,\,
\forall x \in K^n.
\end{equation*}

\noindent
Let $A_1,A_2, \ldots, A_m \in \mathcal{M}_n(K)$,
such that
$A_\alpha A_\beta = A_\beta A_\alpha$,
$\forall \alpha, \beta \in \{1,2, \ldots, m\}$.

It follows that,
for each pair $(t_0,x_0) \in \mathbb{Z}^m \times K^n$,
the recurrence
\begin{equation}
\label{ecalfaexm2.1}
x(t+1_\alpha) = A_\alpha x(t),
\quad
\forall \alpha \in \{1,2, \ldots, m\},
\end{equation}
with the initial condition $x(t_0)=x_0$, has unique solution
$$
x \colon \big\{ t \in \mathbb{Z}^m \, \big|\,
t \geq t_0 \big\} \to K^n,
$$
\begin{equation}
\label{ecalfaexm2.2}
x(t)=A_1^{(t^1-t^1_0)}A_2^{(t^2-t^2_0)}
\cdot
\ldots
\cdot
A_m^{(t^m-t^m_0)}x_0.
\end{equation}

If, for each $\alpha \in \{1,2, \ldots, m\}$,
the matrix $A_{\alpha}$ is invertible,
then there exists a unique function
$\tilde{x} \colon \mathbb{Z}^m \to K^n$,
solution of the recurrence {\rm (\ref{ecalfaexm2.1})},
with $\tilde{x}(t_0)=x_0$; the function
$\tilde{x}(\cdot)$ is a unique extension
of $x(\cdot)$ and it is defined by the formula {\rm (\ref{ecalfaexm2.2})},
but for each $t\in \mathbb{Z}^m$.

\section*{Acknowledgments}

The work has been funded by the Sectoral Operational Programme Human Resources
Development 2007-2013 of the Ministry of European Funds through
the Financial Agreement POSDRU/159/1.5/S/132395.

Partially supported by University Politehnica of Bucharest and by Academy of Romanian Scientists.
Special thanks goes to Prof. Dr. Ionel \c Tevy, who was willing to participate in our discussions
about multivariate sequences and to suggest the title ``multiple recurrences".

\end{document}